\documentclass[12pt]{amsart}

\usepackage{amssymb,amsmath,amsthm, url}
\usepackage{comment}
\usepackage{enumitem}
\usepackage[all]{xy}
\usepackage{datetime}

\setenumerate{label=(\alph*), font=\normalfont}

\newtheorem{thm}[equation]{Theorem}
 \newtheorem{prop}[equation]{Proposition}
 \newtheorem{lem}[equation]{Lemma}
 \newtheorem{cor}[equation]{Corollary}

 \newtheorem{reduction}[equation]{Reduction}

 \theoremstyle{definition}

 \newtheorem{remark}[equation]{Remark}

\numberwithin{equation}{section}

\newcommand{\bbZ}{{\mathbb{Z}}}

\newcommand{\bbP}{{\mathbb{P}}}

\newcommand{\bbQ}{{\mathbb{Q}}}

\newcommand{\Mat}{\operatorname{M}}

\newcommand{\Ind}{\operatorname{Ind}}
\newcommand{\Res}{\operatorname{Res}}
\newcommand{\Mod}{\operatorname{Mod}}

\newcommand{\ed}{\operatorname{ed}}

\newcommand{\trdeg}{\operatorname{trdeg}}
\newcommand{\Gal}{\operatorname{Gal}}

\newcommand{\Char}{\mathop{\mathrm{char}}\nolimits}

\newcommand{\End}{\operatorname{End}}

\renewcommand{\phi}{\varphi}

\author{Dave Benson}

\address
{Institute of Mathematics \\
University of Aberdeen \\
King's College \\
Aberdeen AB24 3UE \\
Scotland, UK}

\email{d.j.benson@abdn.ac.uk}

\thanks{The research of the first author was supported by 
the Collaborative Research Group in Geometric and Cohomological
Methods in Algebra at the Pacific Institute for the Mathematical
Sciences, Vancouver, Canada (2016)}

\author[Zinovy Reichstein]{Zinovy Reichstein}

\address
{Department of Mathematics \\
University of British Columbia \\
Vancouver
\\
CANADA}

\email{reichst@math.ubc.ca}

\date{\today\ at \currenttime}

\thanks
{The second author was partially supported by NSERC Discovery Grant 250217-2012}

\begin{document}

\title[Fields of definition for representations]{Fields of definition 
for representations of associative algebras}


\begin{abstract}
We examine situations, where representations of a finite-dimensional
$F$-algebra $A$ defined over a separable extension field
$K/F$, have a unique minimal field of definition. Here 
the base field $F$ is assumed to be a $C_1$-field.
In particular, $F$ could be a finite field or $k(t)$ or $k((t))$,
where $k$ is algebraically closed.

We show that a unique minimal field of definition exists if
(a) $K/F$ is an algebraic extension or (b) $A$ is 
of finite representation type. Moreover, in these situations
the minimal field of definition is a finite extension of $F$.
This is not the case if $A$ is of infinite representation type 
or $F$ fails to be $C_1$.
As a consequence, we compute the essential dimension of the functor
of representations of a finite group, generalizing a theorem
of N.~Karpenko, J.~Pevtsova and the second author.
\end{abstract}

\maketitle

\section{Introduction}

\subsection*{Notational conventions}
Throughout this paper $F$ will denote a base field and
$A$ a finite-dimensional associative algebra over $F$. 
If $K/F$ is a field extension (not necessarily algebraic), 
we will denote the tensor product $K \otimes_F A$ by $A_K$. 
Let $M$ be an $A_K$-module. Unless otherwise specified, we will
always assume that $M$ is finitely generated (or equivalently, 
finite-dimensional as a $K$-vector space). If $L/K$ is a field extension,
we will write $M_L$ for $L \otimes_K M$.

An intermediate field $F \subset K_0 \subset K$ is called
a \emph{field of definition} for $M$ if
there exists a $K_0$-module $M_0$ such that $M \cong (M_0)_K$. 
In this case we will also say that $M$ \emph{descends} to $K_0$.

\subsection*{Minimal fields of definition}

A field of definition $K_0$ of $M$ is said to be \emph{minimal}
if whenever $M$ descends to a field $L$ with $F \subset L \subset K$,
we have $K_0 \subset L$.

Minimal fields of definition do not always exist. For example,
let $F=\bbQ$ and $A$ be the quaternion algebra
\[ A=\bbQ \{ i, j ,k \} /(i^2=j^2=k^2=ijk=-1). \]
Then $A_K$ has a two dimensional module $M$ given by
\[ i\mapsto \begin{pmatrix} a&b\\b&-a\end{pmatrix},\qquad
j \mapsto \begin{pmatrix}b&-a\\-a&-b\end{pmatrix},\qquad
k\mapsto \begin{pmatrix}0&1\\-1&0\end{pmatrix} \]
over any field $K$ of characteristic $0$ having two elements 
$a$ and $b$ such that $a^2+b^2=-1$. 
Examples of such fields include $\mathbb C$, $\bbQ(\sqrt{-1})$ or 
$\bbQ(\sqrt-5)$.  If we take $K$ to be ``the generic field" of
this type, i.e., the field of fractions of $\bbQ[a, b]/(a^2+ b^2+1)$, 
then $M$ has no minimal field of definition; see
Proposition~\ref{prop.quaternion}(b).

\subsection*{$C_1$-fields}
Such examples arise because of the existence of noncommutative
division rings of finite dimension over $F$. So, it makes sense to develop
a theory over fields for which these do not exist. We say that $F$ is a
$C_1$-field if any homogeneous polynomial $f_1(x_1, \dots, x_n)$ 
of degree $d < n$ with coefficients in $F$ has a
non-trivial solution in $F^n$.
Examples of $C_1$-fields include finite fields, $k(t)$, and $k((t))$, where
$k$ is algebraically closed.  An algebraic extension of 
a $C_1$-field is again $C_1$.  Over a $C_1$-field every
every central division algebra is commutative.
For a detailed discussion of this class of fields, including proofs 
of the above assertions, we refer the reader to~\cite[Section 6.2]{gs}. 
Our first main result is as follows.

\begin{thm} \label{thm.main1} 
Let $F$ be a $C_1$-field, $A$ be a finite-dimensional $F$-algebra, 
$K/F$ be a separable algebraic field extension 
and $M$ be an an $A_K$-module.  Then $M$ has a minimal field 
of definition $F \subset K_0 \subset K$ such that $[K_0:F] < \infty$.
\end{thm} 

To illustrate Theorem~\ref{thm.main1}, let us consider a simple case,
where $\Char(F) = 0$, $A := FG$ is
the group algebra of a finite group $G$, and $M$ is absolutely 
irreducible $KG$-module. Denote the character of $G$ associated to $M$
by $\chi \colon G \to K$.
We claim that in this case the minimal field of definition
is $F(\chi)$, the field generated over $F$ by the character values 
$\chi(g)$, as $g$ ranges over $G$. 
Indeed, it is clear that $F(\chi)$ has to be contained in 
any field of definition $F \subset K_0 \subset K$ of $M$. 
Thus to prove the above assertion, we only need to show that
$M$ descends to $F(\chi)$.  The minimal degree 
of a finite field extension $L/F(\chi)$, such that $M$ is 
defined over $L$ (i.e., there exists an $LG$-module with character
$\chi$), is the Schur index $s_M$; cf. \cite[Definition 41.4]{cr}.
Thus it suffices to show 
that $s_M = 1$. By~\cite[Theorem (70.15)]{cr}, $s_M$ is the index of the 
endomorphism algebra $\End_A(M)$ of $M$, which is a central simple 
algebra over $F(\chi)$.  Since $F$ is a $C_1$-field, 
and $F(\chi)$ is a finite extension of $F$, $F(\chi)$ is also 
a $C_1$-field. Hence, the index of every central   simple 
algebra over $F(\chi)$ is $1$. In particular, $s_M = 1$, and
$M$ descends to $F(\chi)$, as claimed.

\subsection*{Algebras of finite representation type}

A finite-dimensional $F$-algebra $A$ is said to be {\em of finite 
representation type} if there are only finitely many 
indecomposable finitely generated $A$-modules (up to isomorphism).

Our next result shows that for algebras of finite representation type
Theorem~\ref{thm.main1} remains valid even if the field 
extension $K/F$ is not assumed to be algebraic. 

\begin{thm} \label{thm.main2} 
Let $F$ be a $C_1$-field, $A$ be a finite-dimensional $F$-algebra 
of finite representation type, $K/F$ be a field extension,
and $M$ be an $A_K$-module.  Assume further 
that $F$ is perfectly closed in $K$.  Then $M$ has a minimal field 
of definition $F \subset K_0 \subset K$ such that $[K_0:F] < \infty$.
\end{thm} 

\subsection*{Essential dimension}
Given the $A_K$-module $M$, 
the \emph{essential dimension} $\ed(M)$ of $M$ over $F$ is defined
as the minimal value of the transcendence degree $\trdeg(K_0/F)$,
where the minimum is taken over all fields of definition
$F \subset K_0 \subset K$. The integer $\ed(M)$ may be viewed as
a measure of the complexity of $M$. 
Note that $\ed(M)$ is well-defined, irrespective of whether $M$ 
has a minimal field of definition or not. We also remark that this   
number implicitly depends on the base field $F$, which is assumed 
to be fixed throughout. As a consequence of Theorem~\ref{thm.main2},
we will deduce the following.

\begin{thm} \label{thm.main3}
Let $F$ be a $C_1$-field, $A$ be finite-dimensional $F$-algebra 
of finite representation type, $K/F$ be a field extension,
and $M$ be an $A_K$-module. Then $\ed(M) = 0$. 
\end{thm}

Both Theorem~\ref{thm.main2} and~\ref{thm.main3} fail if 
we do not require $F$ to be a $C_1$-field; see Section~\ref{sect.example}.

\subsection*{The essential dimension of the functor of $A$-modules}
We will also be interested in the essential dimension 
$\ed(\Mod_A)$ of the functor $\Mod_A$ from the category 
of field extensions of $F$ to the category of sets, which
associates to a field $K$, the set of isomorphism classes of $A_K$-modules. 
By definition,  
\[ \ed(\Mod_A) := \sup \; \ed(M) \, , \]
where the supremum is taken over all field extensions $K/F$ 
and all finitely generated $A_K$-modules $M$. The value
of $\ed(\Mod_A)$ may be viewed as 
a measure the complexity of the representation theory of $A$.
For generalities on the notion of essential dimension 
we refer the reader to~\cite{bf, reichstein1, reichstein2, 
merkurjev1, merkurjev2}. 
Essential dimensions of representations of finite groups 
and finite-dimensional algebras are studied in~\cite{krp} 
and \cite[Section 3]{bdh}.

Note that while $\ed(M) < \infty$, for any given $A_K$-module $M$ 
(see Lemma~\ref{lem.fin-gen}), $\ed(\Mod_A)$ may be infinite. In
particular, in the case, where $A = FG$ is the group algebra of 
a finite group $G$ over a field $F$, it is shown in~\cite[Theorem~14.1]{krp}
that $\ed(\Mod_A) = \infty$, provided that $F$ is a field of
characteristic $p > 0$ and $G$ has a subgroup isomorphic 
to $(\bbZ/p \bbZ)^2$. Our final main result is the following
amplification of~\cite[Theorem~14.1]{krp}.

\begin{thm} \label{thm.main4}
Let $G$ be a finite group and
$F$ be a field of characteristic $p$. Then
the following conditions are equivalent:
\begin{itemize}
\item[\rm (1)] The $p$-Sylow subgroup of $G$ is cyclic, 
\item[\rm (2)] $\ed(\Mod_{FG}) = 0$, 
\item[\rm (3)] $\ed(\Mod_{FG}) < \infty$.
\end{itemize}
\end{thm}

\section{Preliminaries on fields of definition}

\begin{lem} \label{lem.fin-gen}
Let $A$ be a finite-dimensional $F$-algebra, $K/F$ be a field extension
and $M$ be an $A_K$-module. Then $M$ descends to an intermediate subfied
$F \subset E \subset K$, where $E/F$ is finitely generated.
\end{lem}

\begin{proof}
Suppose $a_1, \dots, a_r$ generate $A$ as an $F$-algebra.
Choose an $F$-vector space basis for $M$. Then
the $A$-module structure of $M$ is completely determined by
the matrices representing multiplication by $a_1, \dots, a_r$ in
this basis.  Each of these matrices has $n^2$ entries in $K$, where
$n = \dim_F(M)$.  Let $E \subset K$ be the field extension of $F$ obtained 
by adjoining these these $rn^2$ entries to $F$. Then $M$ descends 
to $E$. 
\end{proof}

Next we recall the classical theorem of Noether and Deuring.
For a proof, see~\cite[(29.7)]{cr} or \cite[Lemma 5.1]{bp}.

\begin{thm} \label{thm.nd} {\rm(Noether-Deuring Theorem)}
Let $K/E$ be a field extension, $A$ be a finite-dimensional $E$-algebra,
and $M_1$, $M_2$ and $M$ be $A$-modules.
If $K \otimes_E M_1$ and $K \otimes_E M_2$ are isomorphic 
as $A_K$-modules, then $M_1$ and $M_2$ are isomorphic as $A$-modules.
\qed
\end{thm}

\begin{lem} \label{lem.linear}
Let $F$ be a field, $A$ be a finite-dimensional $F$-algebra, 
$F \subset E \subset K$ be a field extension, $N$ be $A_E$-module,
and $M = N_K$ and $F \subset E_0 \subset E$ be an intermediate field.
Then  

\smallskip
(a) $M$ descends to $E_0$ if and only if $N$ descends to $E_0$.

\smallskip
(b) If $F \subset E_{min} \subset K$ is a minimal field of definition for $M$, 
then $E_{min}$ is a minimal field of definition for $N$.
\end{lem}

\begin{proof} 
(a) If $N$ descends to $E_0$, then clearly so does $M$. Conversely, 
suppose $M$ descends to $E_0$.
That is, there exists a $E_0$-module $N_0$ such that 
$K \otimes_{E_0} N_0 \simeq M$ as an $A_K$-module. Consider the
$A_E$-modules $N_1 := E \otimes_{E_0} N_0$ and $N_2 := N$.
Both become isomorphic to $M_K$ over $K$. By Theorem~\ref{thm.nd},
$N_1 \simeq N_2$ as $A_E$-modules. In other words, $N$ descends to $E_0$,
as desired.

\smallskip
(b) Since $E$ is a field of definition for $M$, we have $E_{min} \subset E$.
By part (a), $E_{min}$ is a field of definition for $N$, and part (b) follows.
\end{proof}

We finally come to the main result of this section.

\begin{prop} \label{prop.M^n} Suppose $F$ is a $C_1$-field, $A$ is a
finite-dimensional $F$-algebra, $K/F$ is a field extension,
$M$ is a finitely generated $A_K$-module, and
$F \subset K_0 \subset K$ is an intermediate field, such 
that $[K_0 : F] < \infty$.

If $M^n$ is defined over $K_0$ for some positive integer $n$, 
then so is $M$.
\end{prop}

\begin{proof} 
Set $\End^{ss}_{A_K}(M)$ to be the quotient of $\End_{A_K}(M)$ 
by its Jacobson radical. 
By our assumption $M^n \simeq K \otimes_{K_0} N$ for some $A_{K_0}$-module $N$.
By Fitting's Lemma, \[ \End^{ss}_{A_K}(M^n) \simeq \Mat_n(D), \]
where $D$ is a finite-dimensional division algebra over some finite field
extension $K'$ of $K$. On 
the other hand,
\begin{equation} \label{e.M^n}
\Mat_n(D) \simeq \End^{ss}_{A_K}(M^n) \simeq 
\End^{ss}_{A_K}(K \otimes_{K_0} N) \simeq 
K \otimes_{K_0} \End^{ss}_{A_{K_0}}(N) \, . 
\end{equation}
We conclude that $\End^{ss}_{A_{K_0}}(N)$ is a simple algebra over $K_0$,
i.e., 
\begin{equation} \label{e.End}
\End^{ss}_{A_{K_0}}(N) \simeq M_m(D_0) 
\end{equation}
over $K_0$, for some integer $m \geqslant 0$ and some 
finite-dimensional central division algebra $D_0$ over 
a field $K_0'$ such that $K_0'$ is a finite extension of $K_0$.
Now recall that we are assuming that $F$ is a $C_1$-field and 
\[ F \subset K_0 \subset K_0' \]
are finite field extensions. Hence, $K_0'$ is also a $C_1$-field, and thus 
every finite-dimensional division algebra over $K_0'$ 
is commutative. In particular, $D_0 = K_0'$, is a field, and 
\[ \Mat_n(D) \simeq K \otimes_{K_0} \End^{ss}_{A_{K_0}}(N) \simeq
K \otimes_{K_0} \Mat_m(K_0') \, .   \]
Since $\Mat_n(D)$ is a simple algebra, we conclude that 
$K \otimes_{K_0} K_0'$ is a field. Moreover, the index of 
$\Mat_m(K \otimes_{K_0} K_0')$ is $1$; hence, $D = K'$ 
is commutative, $K \otimes_{K_0} K_0' = K'$, and $m = n$. 

Now~\eqref{e.End} tells us that $N \simeq M_0^n$ as a $A_{K_0}$-module,
for some indecomposable $A_{K_0}$-module $M_0$. 
Since $K \otimes_{K_0} N \simeq M^n$, by 
the Krull-Schmidt theorem $K \otimes_{K_0} M_0 \simeq M$.
Thus $M$ descends to $K_0$, as claimed.
\end{proof}   

\section{Proof of Theorem~\ref{thm.main1}}

We begin with a simple criterion for the existence of a minimal field of
definition.

\begin{lem} \label{lem.reduction}
Let $A$ be a finite-dimensional $F$-algebra, and $K/F$ be a field extension,
and $M$ be an $A_K$-module, satisfying conditions (a) and (b) below.
Then $M$ has a minimal field of definition.

\smallskip
(a) Suppose $M$ descends to an intermediate field 
$F \subset L \subset K$, i.e., $M \simeq K \otimes_L N$ for 
some $A_L$-module $N$.
Then $N$ further descends to a subfield $F \subset E \subset L$, 
where $[E : F] < \infty$.  

\smallskip
(b) Suppose $M$ descends to an intermediate field $F \subset E \subset K$
such that $[E:F] < \infty$. That is, $M \simeq K \otimes_{E} N$ for
some $A_E$-module $N$. Then $N$ has a minimal field of definition
$E_{min} \subset E$.
\end{lem}

\begin{proof} Condition (a) implies that $M$ is defined over some 
$F \subset E \subset K$ with $[E:F] < \infty$. Let the $A_E$-module
$N$ and the field $E_{min} \subset E$ be as in (b).

We claim that $E_{min}$ is independent of the choice of $E$. That is,
suppose $F \subset E' \subset K$ is another field of definition 
of $M$ with $[E':F] < \infty$,
$M := K \otimes_{E'} N'$ for some $A_{E'}$-module $N'$.  
Let $E'_{min} \subset E'$ be the minimal field of
definition of $N'$, so that $N' := E' \otimes_{E'_{min}} N'_{min}$. 
Then our claim asserts that $E_{min} = E_{min}'$. 
If we can prove this claim, then clearly $E_{min}$ is the 
minimal field of definition for $M$. 
Our proof of the claim will proceed in two steps.

First assume $E \subset E'$. By Lemma~\ref{lem.linear}(b), 
$E'_{min}$ is a minimal field of definition for $N$. By uniqueness 
of the minimal field of definition for $N$, $E_{min} = E'_{min}$.  

Now suppose $F \subset E \subset K$ and $F \subset E' \subset K$ are 
fields of definition for $M$ such that $[E:F] < \infty$ and $[E':F] < \infty$.
Let $E''$ be the composite of $E$ and $E'$ in $K$ and
$E''_{min}$ be the minimal field of definition of $N_{E''} \simeq N'_{E''}$.
(Note that $N_{E''}$ and $N'_{E''}$ become isomorphic over $K$; hence, 
by Theorem~\ref{thm.nd}, they are isomorphic over $E''$.) 
Then, $[E'' : F] < \infty$,
and $E, E' \subset E''$. As we just showed, $E_{min} = E''_{min}$  and 
$E'_{min} = E''_{min}$. Thus $E_{min} = E'_{min}$, as desired.
\end{proof}

We now proceed with the proof of Theorem~\ref{thm.main1}.

\begin{reduction} \label{red.finite1}
For the purpose of proving Theorem~\ref{thm.main1}, we may assume without 
loss of generality that 

\smallskip
(a) $K$ is a finite extension of $F$.

\smallskip
(b) $K$ is a Galois extension of $F$.
\end{reduction}

\begin{proof} (a) follows from Lemma~\ref{lem.reduction}.
Indeed, we are assuming that Theorem~\ref{thm.main1} holds whenever 
$K$ is a finite extension of $F$. That is, condition (b) 
of Lemma~\ref{lem.reduction} holds. On the other hand, 
condition (a) of Lemma~\ref{lem.reduction} follows from 
Lemma~\ref{lem.fin-gen}.

(b) By part (a), we may assume that
$K/F$ is finite. Let $L$ be the normal closure of $K$ over $F$. 
Then $L/F$ is finite Galois. Lemma~\ref{lem.linear}(b)
now tells us that if
$M_L := L \otimes_K M$ has a minimal field of definition
then so does $M$.
\end{proof}

\begin{lem} Let $F$ be a $C_1$-field, $A$ be 
a finite-dimensional $F$-algebra, $K/F$ 
be a finite Galois extension, and $M$ be an $A_K$-module. The Galois
group $G := \Gal(K/F)$ acts on the set 
of isomorphism classes of $A_K$-modules via
\[ g \colon N \to  {\,}^g N := K  \otimes_{g} N \, . \]
Let $G_M$ be the stabilizer of $M$ under this action.
Then the fixed field $K^{G_M}$ of $G_M$ is the minimal field 
of definition for $M$. 
\end{lem}

\begin{proof}
Suppose $M$ is defined over $K_0$, where $F \subset K_0 \subset K$.
Then clearly ${\,}^g \! M \simeq M$ for every $g \in \Gal(K/K_0)$. Hence, 
$\Gal(K/K_0) \subset G_M$ and consequently, $K^{G_M} \subset K_0$. 
This shows that $K^{G_M}$ is contained in every field of definition 
of $M$.

It remans to show that $M$ descends to $K_0 := K^{G_M}$.
Write $M = M_1^{d_1} \oplus \dots \oplus M_r^{d_r}$, where
$M_1, \dots, M_r$ are distinct indecomposables. The condition 
that ${\, }^g \! M \simeq M$ for any $g \in G_M$ is equivalent to
the following: if $M_j \simeq {\, }^g \! M_i$ for some $g \in \Gal(K/K_0)$,
then $d_i = d_j$. Grouping $G_M$-conjugate indecomposables 
together, we see that $M \simeq S_1 \oplus \dots \oplus S_m$, where
each $S_1, \dots, S_m$ is the $G_M$-orbit sum of one of the
indecomposable modules $M_i$. (Here the orbit sums $S_1, \dots, S_m$ 
may not be distinct.) It thus suffices to show that each 
orbit sum is defined over $K_0$. 

Consider a typical $G_M$-orbit sum 
$S := M_1 \oplus \dots \oplus M_s$, where we renumber  
the indecomposable factors of $M$ so that $M_1, \dots, M_s$ are
the $G_M$-translates of $M_1$.
Let $H$ be the stabilizer of $M_1$ in $G_M$. That is,
\[ H := \{ h \in G_M \, | \, {\, }^h \! M_1 \simeq M_1 \} \, . \]
Let $K_1 : = K^H$. Then 
\[ K \otimes_{K_1} (M_1)_{\downarrow K_1} = \bigoplus_{h \in H} {\, }^h \! M_1
= M_1^{|H|} \, . \]
In particular, this tells us that $M_1^{|H|}$ descends to $K_1$. 
By Proposition~\ref{prop.M^n},
so does $M_1$. In other words, $M_1 \simeq K \otimes_{K_1} N_1$ for some
$K_1$-module $N$. We claim that 
\begin{equation} \label{e.M^n-2}
K \otimes_{K_0} (N_1)_{\downarrow K_0} \simeq S. 
\end{equation} 
If we can prove this claim, then $S$ descends to $K_0$, and we are done.

To prove the claim, note that on the one hand, 
\begin{equation} \label{e.M^n-3}
K \otimes_{K_0} (M_1)_{\downarrow K_0} =  \prod_{g \in G_M}
{\, }^g \! M_1 = S^{|H|} \, . 
\end{equation}
On the other hand, since $M_1 \simeq K \otimes_{K_1} N_1$, we have
\[ (M_1)_{\downarrow K_0} \simeq ((M_1)_{\downarrow K_1})_{\downarrow K_0}  
\simeq (N_1^{|H|})_{\downarrow K_0} \, , \]
and thus 
\begin{equation} \label{e.M^n-4}
K \otimes_{K_0} (M_1)_{\downarrow K_0} =  
(K \otimes_{K_0} ((N_1)_{\downarrow K_0})^{|H|}) \simeq  
(K \otimes_{K_0} (N_1)_{\downarrow K_0}))^{|H|} \, .
\end{equation}
Comparing~\eqref{e.M^n-3} and~\eqref{e.M^n-4}, we obtain
\[ (K \otimes_{K_0} (N_1)_{\downarrow K_0})^{|H|} \simeq S^{|H|} \, . \]
The desired isomorphism~\eqref{e.M^n-2} follows from this by 
the Krull-Schmidt theorem.
\end{proof}

\section{Algebras of finite representation type}

A finite-dimensional $F$-algebra $A$ is said to be {\em of finite 
representation type} if there are only finitely many 
indecomposable finitely generated $A$-modules (up to isomorphism).

\begin{thm} \label{thm.frt} 
Let $F$ be a $C_1$-field, $A$ be finite-dimensional $F$-algebra 
of finite representation 
type, and $K/F$ be a field extension (not necessarily algebraic) 
such that $F$ is perfectly closed in $K$. (That is, for 
every subextension $F \subset E \subset K$ with $[E:F] < \infty$,
$E$ is separable over $F$.) Suppose $M$ is an indecomposable $A_K$-module.
Then

\smallskip
(a) $M$ descends to an intermediate
subfield $F \subset E \subset K$ such that $[E : F] < \infty$.
 
\smallskip
(b) $M$ is a direct summand of $K \otimes_F N$ 
for some indecomposable $A_F$-module $N$.
\end{thm}

\begin{proof} 
(a) Consider the $A$-module $M_{\downarrow \, F}$. 
Generally speaking this module is not finitely generated 
over $A$. Nevertheless, since $A$ has finite representation type, 
thanks to a theorem of Tachikawa~\cite[Corollary 9.5]{tachikawa},
$M_{\downarrow \, F}$ can be written as a direct sum of finitely generated
indecomposable $A$-modules. Denote one of these modules by $N$. That is,
\begin{equation} \label{e2.1}
M_{\downarrow \, F} \simeq N \oplus N' \, , 
\end{equation}
for some $A$-module $N'$ (not necessarily finitely generated).

Let us now take a closer look at $N$. 
By Fitting's lemma, $E := \End^{ss}_A(N)$ is 
a finite-dimensional division algebra over $F$.
Since $F$ is a $C_1$-field, $E$ is a field extension of $F$. 
Now set $F' := E \cap K$ and $m=[F':F]$.  Since $F$ 
is perfectly closed in $K$, $F'$ is finite and separable over $F$. 
Thus \[ \End^{ss}_A(F' \otimes_F N) \simeq F' \otimes_F \End^{ss}_A(N)  
\simeq E \times \dots \times E \, . \]
This tells us that over $F'$, $N$ decomposes into 
a direct sum of $m$ indecomposables,
\begin{equation} \label{e.N_1}
F' \otimes_F N = N_1 \oplus \dots \oplus N_m. 
\end{equation}
By the definition of $F'$, $K \otimes_{F'} E$ is a field. Hence,
each indecomposable $A_{F'}$-module $N_i$ remains indecomposable over $K$.

Tensoring both sides
of \eqref{e2.1} with $K$, we obtain an isomorphism of $A_K$-modules
\begin{align*} 
K \otimes M_{\downarrow \, F} & \simeq (K \otimes_F N) \oplus
(K \otimes_F N') \\
 & = (\bigoplus_{i = 1}^m K \otimes_{F'} N_i) \oplus (K \otimes_F N')  \\
&= (K \otimes_F N_1)  \oplus N'' \, , 
\end{align*}
where $N'' : = (\bigoplus_{i = 2}^m K \otimes_{F'} N_i)
\oplus (K \otimes_F N')$.
Note that \[ K \otimes M_{\downarrow \, F'} \simeq \bigoplus_{B} M \, , \]
where $B$ is a basis of $K$ as an $F'$-vector space. As we mentioned above,
$K \otimes_{F'} N_1$ is an indecomposable $A_K$-module. Since
$K \otimes_{F'} N_1$ is finitely generated and is contained in 
$\bigoplus_{B} M$, 
it lies in the direct sum of finitely many copies of $M$, 
say, in $M^r := M \oplus \dots \oplus M$ ($r$ copies). 
Thus we have maps
\[ K \otimes_F N_1 \hookrightarrow M^r \hookrightarrow \bigoplus_B M \twoheadrightarrow K \otimes_F N_1 \]
whose composite is the identity, and so $K \otimes_F N_1$ is isomorphic
to a direct summand of $M^r$. By the Krull-Schmidt Theorem,  
$K \otimes_{F'} N_1 \simeq M$. In particular, 
$M$ descends to $F'$, as claimed.

\smallskip
(b) By~\eqref{e.N_1}, $N$ is an indecomposable $A$-module, and
$N_1$ is a direct summand of $F' \otimes_F N$. Hence, 
$M \simeq K \otimes_{F'} N_1$ is a direct summand of $K \otimes_F N$,
as desired.
\end{proof}

\begin{cor} \label{cor.frt1}
Let $F$ be a $C_1$-field, $A$ be finite-dimensional $F$-algebra 
of finite representation type, and $K/F$ be a field extension 
such that $F$ is perfectly closed in $K$. Then $A_K$ is also
of finite representation type.
\end{cor}

\begin{proof}
By our assumption $A$ has finitely many indecomposable modules $N^{(1)},
\dots, N^{(d)}$. By Theorem~\ref{thm.frt}(b)
every indecomposable $A_K$-module 
is isomorphic to a direct summand of $K \otimes_F N^{(i)}$ for some $i$. 
By the Krull-Schmidt Theorem, each $K \otimes_F N^{(i)}$ has finitely
many direct summands (up to isomorphism), and the corollary follows.
\end{proof}

\section{Proof of Theorems~\ref{thm.main2} and~\ref{thm.main3}}

We will deduce Theorem~\ref{thm.main2} from 
Lemma~\ref{lem.reduction}. $M$ satisfies condition (b) 
of Lemma~\ref{lem.reduction} by Theorem~\ref{thm.main1}.
It thus remains to show that $M$ satisfies
condition (a) of Lemma~\ref{lem.reduction}.
For notational simplicity, we may
assume that $K = L$ and $M = N$. That is, we want to show that $M$
descends to some intermediate field $F \subset E \subset K$ with
$[E:F] < \infty$. Note that in the case, where $M$ is indecomposable,
this is precisely the content of Theorem~\ref{thm.frt}(a).

In general, write $M = M_1 \oplus \dots \oplus M_r$ as a direct
product of (not necessarily distinct) indecomposables.
By Theorem~\ref{thm.frt}(a), each $M_i$ descends to an intermediate field
$F \subset K_i \subset K$ such that $[K_i : F] < \infty$. 
Let $E$ be the compositum of $K_1, \dots, K_r$ inside $K$. Then 
$[E:F] < \infty$, and $M$ descends to $E$. 
This completes the proof of Theorem~\ref{thm.main2}. \qed

\smallskip
We now proceed with the proof of Theorem~\ref{thm.main3}.
Denote the perfect closure of $F$ in $K$ by $F^{pf}$.
By Theorem~\ref{thm.main2}, $M$ descends to an intermediate field
$F^{pf} \subset K_0 \subset K$ such that $[K_0: F^{pf}] < \infty$.
Hence, $K_0$ is algebraic over $F$, and consequently, $\ed(M) \leqslant
\trdeg_F(K_0) = 0$, as desired.
\qed

\section{An example}
\label{sect.example}

In this section we will show by example that both Theorem~\ref{thm.main2}
and~\ref{thm.main3} fail if we do not require 
$F$ to be a $C_1$-field.  Let $F = \bbQ$
and $A$ be the quaternion algebra 
\[ A=\bbQ \{ x, y \} /(x^2=y^2 = -1, \; xy = - yx). \] 
and $K/F$ be any field having two elements $a$ and $b$ satisfying 
$a^2+b^2=-1$.  Then $A$ has a two dimensional $A_K$-module $M$ given by
\begin{equation} \label{e.quaternion}
x \mapsto \begin{pmatrix}  a & b \\ b & -a \end{pmatrix},\qquad
y \mapsto \begin{pmatrix} b & -a \\ -a & -b \end{pmatrix} \, . 
\end{equation}

\begin{lem} \label{lem.quaternion} The following conditions on
an intermediate field $\bbQ \subset E \subset K$ are equivalent:

\smallskip
(a) $\phi$ descends to $E$, 

\smallskip
(b) $A$ splits over $E$,

\smallskip
(c) there exist elements $a_0$, $b_0 \in E$ such 
that $a_0^2 + b_0^2 = -1$.
\end{lem}

\begin{proof} 
(a) $\Longrightarrow$ (b). Suppose $M$ descends 
to an $A_E$-module $N$.
Since $A_{E} := E \otimes_{\bbQ} A$ is a central simple
$4$-dimensional algebra over $E$, the homomorphism of algebras 
given by
\[ A_E \to \End_E(N) \simeq \Mat_2(E) \]
is an isomorphism. In other words, $E$ splits $A$. 

(b) $\Longrightarrow$ (a). Conversely, suppose $E$ splits $A$. Then 
the representation of $A \to \End_K (M)$ factors as 
follows:
\[ A \to E \otimes_{\bbQ} A \simeq \Mat_2(E) \to \Mat_2(K) 
\, . \]
This shows that $\phi$ descends to $E$. 

The equivalence of (b) and (c) a special case of 
Hilbert's criterion for the splitting of a quaternion algebra; 
see the equivalence of conditions (1) and (7) in~\cite[Theorem III.2.7]{lam} 
as well as Remark (B) on~\cite[p.~59]{lam}. 
\end{proof}

\begin{prop} \label{prop.quaternion}
Let $a$ and $b$ be independent variables over $\bbQ$, 
$E$ be the field of fractions of $\bbQ[a, b]/(a^2+b^2+1)$,
and $M$ be the $2$-dimensional $A_E$-module
given by~\eqref{e.quaternion}. Then

\smallskip
(a) $\ed(M) = 1$,

\smallskip
(b) $M$ does not have a minimal field of definition.
\end{prop}

\begin{proof} (a) The assertion of part (a), follows from
\cite[Example 6.1]{krp}. For the sake of completeness, we will 
give an independent proof.

Suppose $M$ descends to an intermediate
subfield $\bbQ \subset E_0 \subset E$. 
Since $\trdeg_{\bbQ}(E) = 1$,  $\trdeg_{\bbQ}(E_0) = 0$ or $1$.
Our goal is to show that $\trdeg_{\bbQ}(E_0) \neq 0$. Assume 
the contrary, i.e., $E_0$ is algebraic over $\bbQ$.

Note that $E$ is the function field of the conic
curve $a^2 + b^2 + c^2 = 0$ in $\bbP^2$. Since this curve 
is absolutely irreducible, $\bbQ$ is algebraically closed in $E$. 
Thus the only possibility for $E_0$ is $E_0 = \bbQ$, On the other hand,
$M$ does not descend to $\bbQ$ by Lemma~\ref{lem.quaternion}, 
a contradiction.

(b) Suppose $M$ descends to $E_1 \subset E$. Our goal is to show that 
$M$ descends to a proper subfield $E_3 \subset E_1$. 
By Lemma~\ref{lem.quaternion}(c) there exist
$a_1$ and $b_1$ in $E_1$ such that $a_1^3 + b_1^3 = -1$.
If $\bbQ(a_1, b_1)$ is properly
contained in $E_1$, then we are done. Thus we may assume without loss of
generality that $E_1 = \bbQ(a_1, b_1)$. Set $E_3 := \bbQ(a_3, b_3)$. 
where $a_3 := a_1^3 - 3 a_1 b_1^2$ and $b_3 = 3a_1^2 b_1 - b_1^3$. 
We claim that
(i) $A$ splits over $E_3$, and (ii) $E_3 \subsetneq E_1$.

\smallskip
In order to establish (i) and (ii), let us consider the following diagram. 
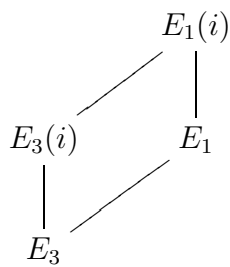
\[ \xymatrix{  &   E_1(i)  \ar@{-}[d]  \\
E_3(i)  \ar@{-}[d] \ar@{-}[ur] & E_1 \\
E_3 \ar@{-}[ur] &  }  \]
Here as usual, $i$ is a primitive $4$th root of $1$.
It is easy to see that $E_1(i) = \bbQ(i)(a_1, b_1) = \bbQ(i)(z)$ is
a purely transcendental extension of $\bbQ(i)$, where
$z = a_1 + b_1 i$ and $\dfrac{1}{z} = - a_1 + b_1 i$. Similarly 
$E_3(i) = \bbQ(i)(z^3)$, where $z^3 = a_3 + b_3 i$ and 
$\dfrac{1}{z^3} = -a_3 + b_3 i$. In particular, this shows 
$a_3^2 + b_3^2 = -1$, thus proving (i). Moreover,
since $z$ is transcendental over $\bbQ(i)$, we have
$[E_1(i): E_3(i)] = [\bbQ(i)(z): \bbQ(i)(z^3)] = 3$ and thus
\[ [E_1: E_3] = \frac{[E_3(i): E_3] \cdot [E_1(i): E_3(i)]}{[E_1(i): E_1]} =
\frac{2 \cdot 3}{2} = 3 \, . \]
This proved (ii).
\end{proof}

\begin{remark} Write $z^n = a_n + b_ni$ for suitable $a_n, b_n \in E_1$
and set $[E_1:E_n] = n$. We showed above that $[E_1:E_3] = 3$ and thus
$E_3 \subsetneq E_1$.  The same argument yields  $[E_1:E_n] = n$ 
for any positive integer $n$.
\end{remark}

\section{Proof of Theorem~\ref{thm.main4}}

We shall actually prove a stronger, more natural theorem, about
blocks of finite group algebras. Theorem \ref{thm.main4} will follow
from the fact that $p$-Sylow $p$ of a finite group $G$ are cyclic
if and only if every block over a field $F$ of characteristic $p$ 
has cyclic defect.

\begin{thm} \label{thm.block}
Let $B$ be a block of a finite group algebra $FG$, where $F$
is a field of characteristic $p$. Then the following are equivalent:
\begin{itemize}
\item[\rm (1)] $B$ has cyclic defect,
\item[\rm (2)] $\ed(\Mod_B)=0$,
\item[\rm (3)] $\ed(\Mod_B)<\infty$.
\end{itemize}
\end{thm}

The implication (1) $\Longrightarrow$ (2) is a direct consequence
of Theorem~\ref{thm.main3}. The implication (2) $\Longrightarrow$ (3) 
is obvious. 

The remainder of this section will be devoted to proving that
(3) $\Longrightarrow$ (1). 
We shall show that if $B$ has non-cyclic defect, then
$\ed(\Mod_B)=\infty$. Let $K$ be an extension field of $F$, 
let $e$ be the block idempotent of $B$,
let $D$ be a defect group of $B$,
and let $N=\Phi(D)$, the Frattini subgroup of $D$.
If $D$ is not cyclic, $D/N$ is elementary abelian of rank $r\ge 2$,
with basis the images of elements $g_1,\dots,g_r\in D$.
Since $D$ is a defect group of $B$, any $KD$-module $M$
is a summand of $\Res_{G,D}(e.\Ind_{D,G}(M))$.

Now let $n>0$, and let $K=F(t_{1,1},\dots,t_{n,r})$ be a function field in $nr$
indeterminates, and let $M_i$ $(1\le i\le n)$ be the two dimensional
$KD$-module
\[ g_j \mapsto \begin{pmatrix} 1 & t_{i,j} \\ 0 & 1 \end{pmatrix}. \]
Then $J^2(KD)$ is in the kernel of $M_i$, so $M_i$ is really a module
for $KD/J^2(KD)$, which has a basis $1, (g_1-1),\dots,(g_r-1)$. 
The last $r$ elements of this list form a basis for $J(KD)/J^2(KD)$,
and we form a vector space $V$ with basis $(g_1-1),\dots,(g_r-1)$.
The kernel of $M_i$ as a module for $KD/J^2(KD)$ is the codimension 
one subspace $H_i$ of \[ J(KD)/J^2(KD)\cong V \] 
given by
\begin{equation} \label{e.H_i}
H_i : = \{\lambda_j(g_j-1) \mid \sum_j t_{i,j}\lambda_j=0\}. 
\end{equation}
By the Mackey decomposition theorem, 
the module $M'_i=\Res_{G,D}(e.\Ind_{D,G}(M_i))$ is a 
direct sum of at least one copy of $M_i$, 
some conjugates of $M_i$ by elements of $N_G(D)$, 
and some modules of the form 
$\Ind_{D\cap {^g}D,D}\Res_{{^g}D,D\cap {^g}D}{^g}M$.
It follows that the Jordan canonical form of elements of $V$
on $M'_i$ is constant, except on a set $S_i$, which is a finite 
union of hyperplanes $N_G(D)$-conjugates of $H_i$ and linear 
subspaces of smaller dimension. 

Now let $M :=\bigoplus_i M_i$. Our goal is to show that
\[ \ed(e.\Ind_{D,G}(M)) \geqslant n (r - 1) \, . \]
This will imply that $\ed(\Mod_B)\ge n(r-1)$ for every $n>0$ and 
thus $\ed(\Mod_B)=\infty$, as desired.

Note that $e.\Ind_{D,G}(M)$ is a module
whose restriction to $D$ is $\bigoplus_i M'_i$. If $e.\Ind_{D,G}(M)$
descends to an intermediate subfield $F \subset K_0 \subset K$, then 
so does the set $\bigcup_i S_i\subset V$ and its natural image in
$\bbP(V) = \bbP^{r-1}$, which we will denote by $S$.
To complete the proof of Theorem~\ref{thm.block},  it remains to
show that if $S$ descends to $K_0$, then 
\begin{equation} \label{e.block}
\trdeg_F(K_0) \geqslant n (r - 1) \, . 
\end{equation}  

\begin{lem} \label{lem.subvariety} 
Let $S \subset \bbP^{r-1}$ be a projective variety defined over a field
$K$.  Assume that a hyperplane $H$ given by 
$a_1 x_1 + a_2 x_2 + \dots + a_r x_r = 0$ is an irreducible 
component of $S$ for some $a_1, \dots, a_r \in K$ (not all zero).
Suppose $S$ descends to a subfield $K_0 \subset K$. 
Then each ratio $a_j/a_l$ is algebraic over $K_0$, as long as $a_l \neq 0$.
\end{lem} 

To deduce the inequality~\eqref{e.block} from Lemma~\ref{lem.subvariety}, 
recall that in our case $S$ is the union of the hyperplanes $H_1, \dots, H_n$,
a finite number of other hyperplanes (translates of $H_1, \dots, H_n$
by elements of $N_G(D)$) and lower-dimensional linear subspaces 
of $\bbP(V) = \bbP^{r-1}$.
In the basis $(g_1 - 1), \dots, (g_r - 1)$ of $V$, 
$H_i$ is given by $t_{i, 1} x_1 + t_{i, 2} x_2 + \dots + t_{i, r} x_r = 0$;
see~\eqref{e.H_i}.  Thus by Lemma~\ref{lem.subvariety} 
the elements $t_{i,j}/t_{i, 1}$ are algebraic
over $K_0$ for every $i = 1, \dots, n$ 
and every $j = 2, \dots, r$. In other words, if $K_1$ is the algebraic
cosure of $K_0$ in $K$, then each $t_{i,j}/t_{i, 1} \in K_1$, and thus
$\trdeg_F(K_0) = \trdeg_F(K_1) \geqslant n (r-1)$, as desired.

\begin{proof}[Proof of Lemma~\ref{lem.subvariety}] 
We may assume without loss of generality that
$K_0$ is algebraically closed. To reduce to this case, we
replace $K_0$ by its algebraic closure $\overline{K_0}$ and $K$ by
a compositum of $K$ and $\overline{K_0}$. If we know that
each $a_{i, j}$ is algebraic over $\overline{K_0}$ (or equivalently,
is contained in $\overline{K_0}$), then $a_{i, j}$ is algebraic over $K_0$.

Now assume that $K_0$ is algebraically closed. Since $S$ is defined over $K_0$,
every irreducible component of $S$ is defined over $K_0$. In particular,
$H$ is defined over $K_0$. That is, the point $(a_1: \dots : a_r)$
of the dual projective space $\check{\bbP}^{r-1}$ 
is defined over $K_0$. Equivalently, $a_i/a_j \in K_0$ whenever $a_l \neq 0$.
This completes the proof of the claim and thus
of Lemma~\ref{lem.subvariety} and Theorem~\ref{thm.block}. 
\end{proof}

\section*{Acknowledgements}
The second author would like to thank Julia Pevtsova for helpful discussions.

\end{document}